\documentclass[11pt]{amsart}

\usepackage{caption}
\DeclareCaptionLabelFormat{period}{Table \thetable.}
\usepackage[matrix,arrow,curve]{xy}
\usepackage{amsmath}
\usepackage{amssymb}
\usepackage{amsthm}
\usepackage{stmaryrd}
\usepackage[matrix,arrow,curve]{xy}
\usepackage{bbm}
\usepackage{float}
\usepackage{graphicx}

\usepackage[colorlinks=true]{hyperref}

\renewcommand{\L}{\mathcal{L}}
\newcommand{\A}{\mathbb{A}}
\newcommand{\K}{\mathcal{K}}
\newcommand{\B}{\mathbb{B}}
\renewcommand{\O}{\mathcal{O}_E}
\newcommand{\F}{\mathcal{F}}
\newcommand{\G}{\mathcal{G}}
\newcommand{\E}{\mathcal{E}}
\renewcommand{\P}{\mathbb{P}}
\newcommand{\Z}{\mathbb{Z}}

\renewcommand{\dim}{\operatorname{dim}}
\newcommand{\Hom}{\operatorname{Hom}}
\newcommand{\Ext}{\operatorname{Ext}}
\newcommand{\RHom}{\operatorname{\textbf{R}Hom}}
\newcommand{\rk}{\operatorname{rk}}
\newcommand{\Proj}{\operatorname{Proj}}
\newcommand{\Coh}{\operatorname{Coh}}
\newcommand{\syz}{\operatorname{syz}}
\newcommand{\cosyz}{\operatorname{cosyz}}
\newcommand{\st}{\operatorname{st}}
\newcommand{\MCM}{\underline{\operatorname{MCM}}}
\newcommand{\tors}{\operatorname{tors}}
\newcommand{\cone}{\operatorname{cone}}
\newcommand{\CR}{\operatorname{CR}}
\newcommand{\coker}{\operatorname{coker}}
\newcommand{\depth}{\operatorname{depth}}
\newcommand{\rad}{\operatorname{rad}\langle -, - \rangle}
\newcommand{\ev}{\operatorname{ev}}
\newcommand{\adj}{\operatorname{adj}}

\newtheorem{Lem}{Lemma}[section]
\newtheorem{Prop}[Lem]{Proposition}
\newtheorem{Cor}[Lem]{Corollary}
\newtheorem{Thm}[Lem]{Theorem}
\newtheorem*{Rem}{Remarks}

\def\gr{\operatorname{gr}\!}
\def\grproj{\operatorname{grproj}\!}

\title{Betti Tables of MCM Modules Over the Cone of a Plane Cubic}
\author{Alexander Pavlov}
\address{MSRI, 17 Gauss Way, Berkeley, CA, 94720, USA}
\email{apavlov@msri.org}
\date{}



\begin{document}

\begin{abstract}
We show that for maximal Cohen-Macaulay modules over a homogeneous coordinate rings of smooth Calabi-Yau varieties $X$ computation of Betti numbers can be reduced to computations of dimensions of certain $\Hom$ groups in the bounded derived category $D^b(X)$.

In the simplest case of a smooth elliptic curve $E$ imbedded into $\P^2$ as a smooth cubic we use our formula to get explicit answers for Betti numbers. Description of the automorphism group of the derived category $D^b(E)$ in terms of the spherical twist functors of Seidel and Thomas plays a major role in our approach. We show that there are only four possible shapes of the Betti tables up to a shifts in internal degree, and two possible shapes up to a shift in internal degree and taking syzygies.
\end{abstract}

\maketitle

\tableofcontents

\section{Introduction}

Elliptic curves is a vast and old topic in algebraic geometry, mostly famous by its arithmetical facet. This paper concentrated on another aspect, namely on the category of graded modules over the cone of a smooth elliptic curve. This category is deeply related to the category of coherent sheaves on the elliptic curve itself. By definition homogeneous coordinate ring is
$$
R_E=\bigoplus_{i \geq 0} H^0(E, \L^{\otimes i}),
$$
where $\L$ is a very ample line bundle, sections of $\L$ provide the embedding of $E$ into $\P^2$. If $\F$ is a coherent sheaf Serre's functor $M=\Gamma_*(\F) = \bigoplus_{k \geq 0} H^0(E, \F(k))$ associate graded module $M$ over $R_E$. It is known classically that $M$ is maximal Cohen-Macaulay module (MCM) for a vector bundle over a curve, in higher dimensions vector bundles with this property a called arithmetically maximal Cohen-Macaulay modules. 

We call $\operatorname{Spec}(R_E)$ the cone over the elliptic curve $E$. It is a singular affine surface, with an isolated singularity at the irrelevant ideal of $R_E$\,. We started with a smooth curve $E$, but the cone over $E$ is a singular space, vector bundles and more generally derived category of coherent sheaves over $E$ are naturally related to MCM modules over the cone.

Category of MCM modules over a hypersurface $S/(f)$ is equivalent to the category of matrix factorizations. Object of the latter category are ordered pairs of matrices $\{A, B\}$ over $S$ such that $AB=BA=f \operatorname{Id}$, and MCM module corresponding to $\{A, B\}$ is $M=\coker A$. For details and references see corresponding section in the next chapter.

Let us briefly outline he history of the topic of MCM modules over smooth cubics and name several papers. Matrix factorizations of MCM modules of rank $1$ over Fermat cubic $x_0^3+x_1^3+x_2^3=0$ were completely classified in \cite{LPP02}. Classification of rank $1$ MCM modules over the general Hesse cubic $x_0^3+x_1^3+x_2^3 -3\psi x_0 x_1x_2=0$ follows immediately from our results, where the only non-trivial case is given by Moore's matrices. Analogous classification for rank $1$ MCM modules over Weierstrass cubics $y^2 z -x^3 -a x z - bz^3=0$ here obtained more recently in \cite{Galinat14}.

Note that if one of the matrices say $A$ of a matrix factorization $\{A, B\}$ has linear entries and size equal to the degree of the hypersurface $f$ then $\det(A)=f$ and $B=\adj(A)$, where $\adj$ is operation of taking adjugate matrix (transpose of the cofactor matrix). Find such special matrix factorizations is equivalent to finding determinantal presentations of $f$. Some relevant results about determinantal presentations and Betti numbers of line bundles over plane curves could be found in \cite{Beau2000}.

Determinantal presentations of smooth cubics were also studied in \cite{Buckley2006}. Where authors obtained an explicit solution for the Legendre form of the cubic.

Key ingredient of our treatment of Betti numbers of graded modules over $R_E$ is an equivalence of triangulated categories
$$
\Phi : D^b(E) \to \MCM_{\gr}(R_E)
$$
proved by D.Orlov in \cite{Orlov09}. In this paper we call this equivalence of categories Orlov's equivalence. We explain notations and provide some details in the next section. This equivalence is used to formulate questions about graded MCM modules over $R_E$ in terms of coherent sheaves on $E$\,.

The first result (see theorem \ref{main} below) in this direction is immediate.
\begin{Thm}
The graded Betti numbers of an MCM module $M$ are given by
$$
\beta_{i,j}(M)=\dim \Hom_{D^b(E)}(\Phi^{-1}(M), \sigma^{-j}(\Phi^{-1}(k^{st}))[i])\,.
$$
\end{Thm}
For a description of the functor $\sigma$ see the next section. Despite the fact that the formula for Betti numbers above looks more complicated than the standard one
$$
\beta_{i,j}(M)=\dim \Ext^i_{R_E}(M,k(-j)),
$$
it reduces computations of Betti numbers to computations in the bounded derived category $D^b(E)$ of coherent sheaves on $E$\,.

The bounded derived category $D^b(E)$ of a smooth elliptic curve is a relatively simple mathematical object: every complex is formal and there is a description of all indecomposable objects. Computations of dimensions of $\Hom$ spaces 
$$
\beta_{i,j}(M)=\dim \Hom_{D^b(E)}(\Phi^{-1}(M), \sigma^{-j}(\Phi^{-1}(k^{st}))[i])\,.
$$
can be reduced to computations of dimensions of cohomology of vector bundles.

In the hypersurface case the  (quasi-)periodicity of complete resolutions allows us the reduce computations to $\beta_{0,j}$ and $\beta_{1,j}$.

\begin{Thm}
The Betti diagrams of complete resolution of an indecomposable MCM module over the homogeneous coordinate ring $R_E$, up to shift of internal degree and taking syzygies, has one of the following two forms.
\begin{enumerate}
\item The discrete family $F_r$, where $r$ is a positive integer.
$$
\xymatrix{
\ldots& 3r \ar[l]& 3r \ar[l]& 1\ar[l]\\
& & 1 \ar[ul]& 3r \ar[l] \ar[ul]& 3r\ar[l] \ar[ul]& 1 \ar[l]\\
& & & & 1 \ar[ul] \ar[uul]& 3r \ar[l] \ar[ul]& 3r \ar[l] \ar[ul]& \ldots \ar[l]
}
$$
\item The continuous family $G_\lambda(r,d)$\,. Elements in the family $G_\lambda(r,d)$ are parameterized by a pair of integers $(r,d)$, satisfying conditions $r>0$,$d \geq 0$, $3r-2d > 0$, and a point $\lambda \in E$\,.
$$
\xymatrix{
\ldots& d \ar[l]\\
& 3r-2d \ar[ul]& 3r-d \ar[ul] \ar[l]& d \ar[l]\\
& & & 3r-2d \ar[ul]& 3r-d \ar[l] \ar[ul]& \ldots \ar[l]
}
$$
\end{enumerate}
\end{Thm}

Note that in this theorem $F_r$, and $G_\lambda(r,d)$ stand for vector bundles on $E$, the diagrams represent the complete resolutions of the MCM modules that correspond to these vector bundles under Orlov's equivalence.

In commutative algebra various numerical invariants can be associated with a module. Most important invariants include: dimension, depth, multiplicity, rank, minimal number of generators etc. For a graded module over a graded ring values of many invariants can be extracted from the Hilbert polynomial and Hilbert polynomial is easy to compute if we know the Betti numbers.

In our case we compute the Hilbert polynomial and numerical invariants - multiplicity $e(M)$, rank $\rk(M)$ and minimal number of generators $\mu(M)$ of an MCM module $M=\Phi(\F)$ - in terms of the rank $r$ and the degree $d$ of the vector bundle $\F$\,.  The results are summarized in table \ref{NumCubic}.

\section{Preliminaries}
In this chapter we summarize some well known results that we need in this paper. All rings and algebras we consider are Noetherian.

\subsection{Maximal Cohen-Macaulay modules, complete resolutions and matrix factorizations}

In this subsection $R$ is a commutative (local or graded) Gorenstein ring. If $M$ is a finitely generated module over $R$, then depth of $M$ (with respect to the unique maximal ideal in the local case, or the irrelevant ideal in the graded case) is bounded by the dimension of $M$, namely
$$
\depth(M) \leq \dim(M) \leq \dim(R).
$$
If $\depth(M)=\dim(M)$, then the module $M$ is called Cohen-Macaulay, and if $\depth(M)=\dim(R)$ the module $M$ is called a maximal Cohen-Macaulay (MCM). In the following we are going to use the stable category of MCM modules. It means that we consider the full subcategory of MCM modules in the category of finitely generated modules, and in this subcategory morphisms of MCM modules that can be factored through a projective module are considered to be trivial. In other words, in the stable category we have
$$
\underline{\Hom}(M,N)=\Hom(M,N)/\{\text{morphisms that factor through a projective\}}.
$$
In particular, all projective modules (and projective summands) in the stable category are identified with the zero module. Note that the operation of taking syzygies is functorial on the stable category, we denote this functor $\syz$\,. The functor $\syz$ is an autoequivalence of the stable category of MCM modules, the inverse of this functor we denote $\cosyz$\,. Moreover, the stable category of MCM modules admits a structure of triangulated category with autoequivalence $\syz$\,. The details of the construction of the stable category can be found in \cite{Buch87}. The stable category of MCM modules is denoted $\MCM(R)$ in the local case and $\MCM_{\gr}(R)$ in the graded case. The following theorem also can be found in \cite{AusBrid69}.

\begin{Thm}
A finitely generated module $M$ over a Gorenstein ring $R$ is MCM if and only if the following three conditions hold.
\begin{enumerate}
  \item $\Ext^i_R(M,R) \cong 0$, for $i \neq 0$,
  \item $\Ext^i_R(M^*,R) \cong 0$, for $i \neq 0$, where $M^*=\Hom_R(M,R)$ is the dual module.
  \item $M$ is reflexive, this means that the natural map $M \to M^{**}$ is an isomorphism.
\end{enumerate}
\end{Thm}

These three conditions are completely homological and can be used to define an MCM module over a not necessarily commutative (but still Gorenstein) ring $R$\,. From now on we assume that all rings that we consider are Gorenstein.

The above criterion allows us to give a simple construction of a complete projective resolution of an MCM module $M$\,. In this paper we use only complete projective resolutions by finitely generated projectives and we refer to them as complete resolutions. Let $P^\bullet \to M$ be a projective resolution of $M$ and $Q^\bullet \to M^*$ a projective resolution of $M^*$\,. Then we have a projective coresolution of $M^{**} \to (Q^\bullet)^*=\Hom(Q^\bullet,R)$\,. The projective resolution $P^\bullet$ and the coresolution $(Q^\bullet)^*$ can be spliced together
$$
\xymatrix{
P^\bullet \ar[r]& M \ar[d]^\cong \ar[r]& 0 \\
0 \ar[r]& M^{**}  \ar[r]& (Q^\bullet)^* .\\
}
$$
We call the map of complexes $P^\bullet \to (Q^\bullet)^*$ the norm map and we obtain a complete resolution of $M$ as cone of the norm map :
$$
\CR(M)=\cone(P^\bullet \to (Q^\bullet)^*)[-1]\,.
$$
The complex $\CR(M)$ is an unbounded acyclic complex of finitely generated projective $R$-modules. On the other hand, if $K$ is an acyclic unbounded complex of finitely generated projectives then $\coker (K_1 \to K_0)$ is an MCM module. This yields an equivalence of categories

\begin{Thm}
$$
\MCM(R) \cong K^{\text{ac}}(\operatorname{proj}(R)),
$$
where $K^{\text{ac}}(\operatorname{proj}(R))$ stands for the homotopy category of acyclic unbounded complexes of finitely generated projective modules.
\end{Thm}

More generally, if $X^\bullet \in D^b(R)$ is an object of the bounded derived category of $R$, and $r: P^\bullet \to X^\bullet$ is a projective resolution we take $s: Q^\bullet \to (P^\bullet)^*$ to be a projective resolution of $(P^\bullet)^*$\,.  As before, we get a norm map
$$
N: P^\bullet \to (P^\bullet)^{**} \to (Q^\bullet)^*\,.
$$
A complete resolution of $X^\bullet$ is given by the cone of the norm map
$$
\CR(X^\bullet)=\cone(N)[-1]\,.
$$
The complete resolution of $X^\bullet$ is unique up to homotopy.

Just as projective resolutions are used to compute extension groups $\Ext(-,-)$, complete resolutions are used to compute stable extension groups $\underline{\Ext}(-,-)$
$$
\underline{\Ext^i}(X_1^\bullet,X_2^\bullet) = H^i \Hom(\CR(X_1^\bullet),X_2^\bullet)\,.
$$
The next simple lemma is a crucial step in the derivation of our main formula for graded Betti numbers.
\begin{Lem}\label{CR}
If $M$ is an MCM module over $R$, and $N$ is a finitely generated $R$-module, then
$$
\Ext^i(M,N) \cong \underline{\Ext}^i(M,N),
$$
for $i > 0$\,. Moreover, if $N=k$ the above isomorphism is also true for $i=0$,
$$
\Ext^i(M,k) \cong \underline{\Ext}^i(M,k),
$$
for $i \geq 0$\,.
\end{Lem}
\begin{proof}
The complete resolution of $M$ coincides with a projective resolution of $M$ in positive degrees. In the case $N=k$ if we choose a minimal resolution of $M$ the result is immediate.
\end{proof}
We also need the following result about stable extension groups.
\begin{Lem}\label{FGM}
For any finitely generated module $N$ over $R$ there is a finitely generated MCM module $M$ such that there is a short exact sequence
$$
0 \to F \to M \to N \to 0,
$$
where the module $F$ has finite projective dimension. For any $L \in \MCM(R)$ there is an isomorphism of stable extension groups
$$
\underline{\Ext}^i(L,M)\cong \underline{\Ext}^i(L,N),
$$
for $i \in \Z$\,.
\end{Lem}
\begin{proof}
See \cite{Buch87}.
\end{proof}
The module $M$ in the lemma above is called a maximal Cohen-Macaulay approximation or stabilization of $N$\,. We use the latter name and return to this construction in the next section in the context of bounded derived categories. For other properties of stable extension groups see \cite{Buch87}, details on the construction of maximal Cohen-Macaulay approximations can be found in \cite{AB87}.

Now suppose that $R$ is a hypersurface singularity ring $R=S/(w)$, where $S$ is a regular ring, and $w \in S$ defines a hypersurface and $w \neq 0$. In this case for an MCM module $M$ any complete resolution is homotopic to a $2-$periodic complete resolution. Moreover, such $2-$periodic (ordinary and complete) resolutions can be constructed using matrix factorizations of $w$ introduced in \cite{Eisenbud80}. We give a brief outline of the construction.

Let $M$ be an MCM module over $R=S/(w)$\,. If $M$ is considered as an $S$-module, then by Auslander-Buchsbaum formula
$$
\operatorname{pd}_S(M)=\depth(S)-\depth(M),
$$
and so the projective dimension of $M$ as $S$ module is $1$\,. Hence a projective resolution of $M$ over $S$ has the following form
$$
\xymatrix{
0 \ar[r]& P^1 \ar[r]^A& P^0 \ar[r]& M \ar[r]& 0\,.
}
$$
Since multiplication by $w$ annihilates $M$, there is a homotopy $B$ such that the diagram
$$
\xymatrix{
P^1 \ar[r]^A \ar[d]^w & P^0 \ar[d]^w \ar@{-->}[dl]^B \\
P^1 \ar[r]^A & P^0
}
$$
commutes. One can check that the following is true.
$$
AB=BA=w \operatorname{id}\,.
$$
Ordered pairs of maps satisfying the above property are called matrix factorizations of $w$\,. Note that our assumption that $R$ is local or graded implies that the modules $P^0$ and $P^1$ are free, and if we choose a basis in $P^0$ and $P^1$, we can present $A$ and $B$ as matrices. Morphisms in the category of matrix factorizations are defined in a natural way (for details see below), the category of matrix factorizations is denoted $MF(w)$. It is easy to see that the module $M$ can be recovered from a matrix factorization $\{A,B\}$ as
$$
M \cong \coker A,
$$
while the cokernel of the second matrix recovers the first syzygy of $M$ over $R$:
$$
\syz(M) \cong \coker B\,.
$$

In this paper we are interested in graded rings, modules and matrix factorizations, thus from now on all objects are graded. In particular, we assume that $w$ is homogeneous of degree $n$. The category of graded matrix factorizations is denoted $MF_{\gr}(w)$.

A matrix factorization alternatively can be interpreted as a (quasi) $2$-periodic chain $P$ of modules and maps between them
$$
\xymatrix{
\ldots \ar[r] & P^0(-n) \ar[r]^{p^0} & P^1(-n) \ar[r]^{p^1}& P^0 \ar[r]^{p^0} & P^1 \ar[r] & \ldots
}
$$
such that composition of two consecutive maps is not zero (as in the case of chain complexes) but is equal to the multiplication by $w$:
$$
p^0 \circ p^1 = p^1(n) \circ p^0 = \,w.
$$
Abusing notation we still call the maps $p^0$ and $p^1$ differentials. Note that we assume that $p^0$ is of degree $0$, and $p^1$ is of degree $n$. A morphism of matrix factorizations $f: P \to Q$ is an ordered pair $(f^0, f^1)$ of maps of graded free $S$-modules of degree zero commuting with the differentials
$$
\xymatrix{
\ldots \ar[r]& P^0(-n) \ar[r]^{p^0(-n)} \ar[d]^{f^0(-n)} & P^1(-n) \ar[r]^{p^1} \ar[d]^{f^1(-n)} & P^0 \ar[r]^{p^0} \ar[d]^{f^0} & P^1 \ar[r] \ar[d]^{f^1} & \ldots \\
\ldots \ar[r]& Q^0(-n) \ar[r]^{q^0(-n)} & Q^1(-n) \ar[r]^{q^1}& Q^0 \ar[r]^{q^0}& Q^1 \ar[r]& \ldots
}$$

It is more convenient to consider periodic infinite chains
$$
\xymatrix{
\ldots \ar[r] & P^{-1} \ar[r]^{p^{-1}} & P^0 \ar[r]^{p^0}& P^1 \ar[r]^{p^1} & P^2 \ar[r] & \ldots
}
$$
of morphisms of free graded modules such that composition $p^{i+1} p^i$ of any two consecutive maps is equal to multiplication by $w$. Periodicity in the graded cases means that $P^i [2] = P^i(n)$ and $p^i[2]=p^i(n)$, where translation $[1]$ is defined in the same way as for complexes: $P^i[1] = P^{i+1}$ and $p[1]^i = - p^i$. Morphisms of matrix factorizations also satisfy the periodicity conditions $f^{i+2} = f^i(n)$.

The category of matrix factorizations $MF(w)$ satisfies properties similar to the properties of categories of complexes, here we only review some basic definitions and facts, details can be found for example in  \cite{Orlov09}.

For example, the definition of homotopy of morphisms of matrix factorizations is completely analogous to the case of complexes. A morphism $f: P \to Q$ is called null-homotopic if there is a morphism $h: P \to Q[-1]$  such that $f^i = q^{i-1} h^i + h^{i+1} p^i$. The category of matrix factorizations with morphisms modulo null-homotopic morphisms is called the stable (or derived) category of matrix factorizations and is denoted $\underline{MF}_{\gr}(w)$.

For any morphism $f: P \to Q$ from the category $MF_{\gr}(w)$ we define a mapping cone $C(f)$ as an object
$$
\xymatrix{
\ldots \ar[r] & Q^i\oplus P^{i+1} \ar[r]^{c^i} & Q^{i+1}\oplus P^{i+2} \ar[r]^{c^{i+1}} & Q^{i+2}\oplus P^{i+3} \ar[r] & \ldots
}
$$
such that
$$
c^i = \left(
       \begin{matrix}
         q^i & f^{i+1}\\
         0 & -p^{i+1}
       \end{matrix}
     \right)\,.
$$
There are maps $g: Q \to C (f)$ , $g=(id,0)$ and $h: C (f) \to P [1]$ , $h=(0,-id)$.

We define standard triangles in the category $\underline{MF}_{\gr}(w)$ as triangles of the form
$$
\xymatrix{
P \ar[r]^{f} & Q \ar[r]^{g} & C (f) \ar[r]^{h} & P [1]\,.
}
$$
A triangle is called distinguished if it is isomorphic to a standard triangle.

\begin{Thm}
The category $\underline{MF}_{\gr}(w)$ endowed with the translation functor $[1]$ and the above class of distinguished triangles is a triangulated category.
\end{Thm}

Reducing a matrix factorization modulo $w$ and extending it $2$-periodically to the left we get a projective resolution of $M$ over $R$
$$
\xymatrix{
\ldots \ar[r]& \overline{P^0} \ar[r]^{\overline{p^0}}& \overline{P^1} \ar[r]^{\overline{p^1}}& \overline{P^0} \ar[r]& M \ar[r]& 0
\,.}
$$
Extension of this resolution to the right produces a $2$-periodic complete resolution of $M$ over $R$

$$
\xymatrix{
\ldots \ar[r]& \overline{P^0} \ar[r]^{\overline{p^0}}& \overline{P^1} \ar[r]^{\overline{p^1}}& \overline{P^0} \ar[r]^{\overline{p^0}}& \overline{P^1} \ar[r]& \ldots
}
$$
The next theorem follows from the previous discussion and an extension of Eisenbud's result to the graded case.
\begin{Thm}
For a hypersurface ring $R = S/(w)$ there is an equivalence of triangulated categories
$$
\underline{MF}_{\gr}(w) \cong \MCM(R) \cong K^{\text{ac}}(\operatorname{proj}(R)).
$$
\end{Thm}

\subsection{The singularity category and Orlov's equivalence}

Let $R = \bigoplus_{i \geq 0}  R_i$ be a graded algebra over a field $k$\,. We assume that $R_0 = k$, such algebras are called connected. Moreover, we assume that $R$ is Gorenstein, which means that $R$ has finite injective dimension $n$ and that there is an isomorphism
$$
\RHom_R(k,R) \cong k(a)[-n],
$$
where the parameter $a \in \Z$ is called the Gorenstein parameter of $R$\,.
In \cite{Orlov09} D. Orlov described precise relations between the triangulated category $D^b(\operatorname{qgr} (R))$ and the graded singularity category $D^{\gr}_{Sg}(R)$\,. We only use his result for the case $a=0$, and $R$ a homogeneous coordinate ring of some smooth projective Calabi-Yau variety $X$\,.

The abelian category of coherent sheaves $\operatorname{Coh}(X)$ is equivalent, by Serre's theorem, to the quotient category $\operatorname{qgr} (R)$
$$
\operatorname{Coh}(X) \cong \operatorname{qgr} (R),
$$
here the abelian category $\operatorname{qgr} (R)$ is defined as a quotient of the abelian category of finitely generated graded $R$-modules by the Serre subcategory of torsion modules
$$
\operatorname{qgr} (R) = \operatorname{mod}_{\gr} (R) / \operatorname{tors}_{\gr} (R),
$$
where $\operatorname{tors}_{\gr} (R)$ is the category of graded modules that are finite dimensional over $k$. Alternative notation for this category is $\operatorname{Proj} (R) = \operatorname{qgr} (R)$, this notation emphasizes relation of this category with the geometry of the projective spectrum of $R$ for commutative rings, and replaces such geometry for general not necessarily commutative $R$. Details can be found, for example, in \cite{AZ94}.

The bounded derived category of the abelian category of finitely generated graded $R$-modules $D^b(\gr-R) = D^b(\operatorname{mod}_{\gr} (R))$ has a full triangulated subcategory consisting of objects that are isomorphic to bounded complexes of projectives. The latter subcategory can also be described as the derived category of the exact category of graded projective modules, we denote it $D^b(\grproj-R)$. The triangulated category of singularities of $R$ is defined as the Verdier quotient
$$
D^{\gr}_{Sg}(R) = D^b(\gr-R) / D^b(\grproj-R).
$$

In fact, the triangulated category of singularities is one of many facets of the stable category of MCM modules. It is well known (see \cite{Buch87}) that the stable category of MCM modules is equivalent as triangulated category to the singularity category $D^{\gr}_{Sg}(R)$
$$
\MCM_{\gr}(R) \cong D^{\gr}_{Sg}(R)\,.
$$

We have the following diagram of functors
$$
\xymatrix{
& D^b(\gr-R_{\geq i})\ar@/^0.7pc/[dl]^{a} \ar[dr]^{\st}\\
D^b(X) \ar[ur]^{\textbf{R} \Gamma_{\geq i}} && D^{\gr}_{Sg}(R) \ar@/^0.7pc/[ul]^{b}\,.}
$$
In this diagram $D^b(\gr-R_{\geq i})$ denotes the derived category of the abelian category of finitely generated modules concentrated in degrees $i$ and higher. We call the parameter $i \in \Z$ a cutoff. The functor $\textbf{R} \Gamma_{l \geq i}: D^b(X) \to D^b(\gr-R_{\geq i})$ is given by the following direct sum:
$$
\textbf{R} \Gamma_{\geq i}(M)=\bigoplus_{l \geq i} \RHom(\mathcal{O}_X, M(l))\,.
$$
The functor $\st : D^b(\gr-R_{\geq i}) \to D^{\gr}_{Sg}(R)$ is the stabilization functor extended to the bounded derived category $D^b(\gr-R_{\geq i})$. Its construction for any parameter $i$ was given by Orlov. The functor $a$ is an extension of the sheafification functor to the derived category $D^b(\gr-R_{\geq i})$. It is well known and easy to check that the functor $a$ is left adjoint to the functor $\textbf{R} \Gamma_{l \geq i}$\,. Finally, the functor $b$ is a left adjoint functor to the stabilization functor. It can be easily described explicitly if we go to the complete resolutions of MCM modules using the equivalence of triangulated categories
$$
D^{\gr}_{Sg}(R) \cong K^{\text{ac}}(\operatorname{proj}(R)).
$$
We choose a complete resolution of an MCM module and then leave in this complex only free modules generated in degrees $d\geq i$\,. The result is an unbounded complex, but it has bounded cohomology, and thus gives an element in $D^b(\gr-R_{\geq i})$\,.

The special case of Orlov's theorem (see Theorem 2.5 in \cite{Orlov09}) that we need in this paper can be formulated as follows.
\begin{Thm}
Let $X$ be a smooth projective Calabi-Yau variety, $R$ its homogeneous coordinate ring (corresponding to some very ample invertible sheaf). Then for any parameter $i \in \Z$ the composition of functors
$$
\Phi_i = \st \circ \textbf{R} \Gamma_{\geq i} : D^b(X) \to D^{\gr}_{Sg}(R)
$$
is an equivalence of categories.
\end{Thm}

The subtle point is that we get a family of equivalences $\Phi_i$ parameterized by integer numbers. We choose the cutoff parameter $i=0$ and denote $\Phi=\Phi_0$\,.

Abusing notation, we denote the composition of Orlov's and Buchweitz's equivalences by the same $\Phi$\,.
$$
\Phi : D^b(X) \to \MCM_{\gr}(R)\,.
$$
This equivalence of categories is our main tool for computing Betti numbers of MCM modules.

\subsection{Indecomposable objects in the derived category \texorpdfstring{$D^b(E)$}{Db(E)}}

Let $E$ be a smooth elliptic curve, then the category of coherent sheaves $\operatorname{Coh}(E)$ is a hereditary abelian category, namely $Ext^i(\F, \G) \cong 0$ for any coherent sheaves $\F$ and $\G$ and any $i \geq 2$\,. The following theorem attributed to Dold (cf. \cite{Dold60}) shows that if we want to classify indecomposable objects in $D^b(E)$ then it is enough to classify indecomposable objects of $\operatorname{Coh}(E)$\,.
\begin{Thm}
Let $A$ be an abelian hereditary category. Then any object $\F$ in the derived category $D^b(A)$ is formal, i.e., there is an isomorphism
$$
\F \cong \bigoplus_{i \in \Z} H^i(\F)[-i]\,.
$$
\end{Thm}

Indecomposable objects in the abelian category $\operatorname{Coh}(E)$ were classified in \cite{Atiyah57} by M. Atiyah. Every coherent sheaf $\F \in \operatorname{Coh}(E)$ is isomorphic to direct sum of a torsion sheaf and a vector bundle $\F'$:
$$
\F \cong \tors(\F) \oplus \F'\,.
$$
Therefore, it is enough to classify indecomposable torsion sheaves and indecomposable vector bundles. An indecomposable torsion sheaf is a skyscraper sheaf. To describe such sheaf we need two parameters: a point $\lambda \in E$ on the elliptic curve where the skyscraper sheaf is supported, and the degree of that sheaf $d \geq 1$\,. A torsion sheaf with these parameters is isomorphic to $\mathcal{O}_{E, \lambda} /m(\lambda)^{d+1}$\,, where $\lambda \in E$. Isomorphism classes of indecomposable vector bundles of given rank $r \geq 1$ and degree $d \in \Z$ are in bijection with $\operatorname{Pic}^0(E)$\,.

On an elliptic curve $E$ there is a special discrete family of vector bundles, usually denoted by $F_r$ and parameterized by positive integer number $r \in \Z_{>0}$\,. By definition $F_1=\O$, and the vector bundle $F_r$ is defined inductively through the unique non-trivial extension:
$$
0 \to \O \to F_r \to F_{r-1} \to 0\,.
$$
It is easy to see that $\deg(F_r)=0$ and $\rk(F_r)=r$ for $r \in \Z_{>0}$\,. We call $F_r$, $r \in \Z_{>0}$, Atiyah bundles.

At the end of this subsection let us mention that computing cohomology of a vector bundle on an elliptic curve is very simple. In fact we are interested only in the dimensions of the cohomology groups. The following formula for the dimension of the zero cohomology group of an indecomposable vector bundle $\F$ of degree $d=\deg(\F)$ and rank $r=\rk(\F)$ can be found in \cite{Atiyah57}.
$$
\dim H^0(E, \F)=\begin{cases}
$0$, &\text{if $\deg(\F)<0$,} \\
$1$, &\text{if $\deg(\F)=0$ and $\F \cong F_r$,} \\
$d$, &\text{otherwise.} \\
\end{cases}
$$

For the dimensions of the first cohomology groups we use Serre duality in the form $\dim H^1(E, \E) = \dim H^0(E, \E^\vee)$\,. We see that the Atiyah bundles are exceptional cases in the previous formula. Moreover, Atiyah bundles can be characterized as vector bundles on $E$ with degree $d=0$, rank $r \geq 1$, and $\dim H^0(E, F_r)=\dim H^1(E, F_r)=1$\,. They are unique up to isomorphism.

\subsection{Spherical twist functors for \texorpdfstring{$D^b(E)$}{Db(E)}}

We need a description of the automorphism group of the derived category $D^b(E)$\,. Spherical twist functors, introduced in \cite{ST01} by Seidel and Thomas and thus sometimes called Seidel-Thomas twists, play an important role in the description. The reader can find the details and proofs of the statements of this subsection in \cite{BB07} and \cite{HVdB07}.

On a smooth projective Calabi-Yau variety $X$ over a field $k$ an object $\E \in D^b(X)$ is called spherical if
$$
\Hom(\E, \E[*]) \cong H^*(S^{\dim(X)}, k),
$$
where $S^{\dim(X)}$ denotes a sphere of dimension $\dim(X)$\,. The definition of the spherical object can be extended to the non-Calabi-Yau case, but we don't need it in this paper.

With any spherical object $\E$ one can associate an endofunctor $T_{\E}: D^b(X) \to D^b(X)$, called a spherical twist functor. It is defined as a cone of the evaluation map
$$
T_{\E}(\F)=\cone \left( \Hom(\E, \F[*]) \otimes \E \to \F \right)\,.
$$

One of the main results of \cite{ST01} is the following theorem.
\begin{Thm}
If $\E$ is a spherical object in $D^b(X)$ of a smooth projective variety $X$, then the spherical twist
$$
T_{\E}: D^b(X) \to D^b(X)
$$
is an autoequivalence of $D^b(X)$\,.
\end{Thm}

After this short deviation to spherical twists in general on Calabi-Yau varieties we return to the case of an elliptic curve $E$\,. Let us define the Euler form on $\E, \F \in D^b(E)$ as
$$
\langle \E, \F \rangle=\dim H^0(\E, \F)-\dim H^1(\E, \F)\,.
$$

It is easy to see that the Euler form depends only on the classes of $[\E]$ and $[\F]$ in the Grothendieck group $K_0(D^b(E))$\,. Therefore, we treat the Euler form as a form on the Grothendieck group $K_0(D^b(E))$\,. The Riemann-Roch formula can be formulated as the following identity for the Euler form
$$
\langle \E, \F \rangle=\chi(\F) \rk(\E)-\chi(\E) \rk(\F),
$$
where $\chi(\E)$ is the Euler characteristic of $\E$:
$$
\chi(\E)=\langle \O, \E \rangle\,.
$$
Note that $\chi(\E)=\deg(\E)$ for any object $\E$, because the genus $g(E)=1$\,.

Since $E$ is a Calabi-Yau variety, the Euler form is skew-symmetric:
$$
\langle \E, \F \rangle=-\langle \F, \E \rangle,
$$
therefore, the left radical of the Euler form
$$
\operatorname{l.rad}=\{\F \in K_0(D^b(E)) | \langle \F, - \rangle =0 \}
$$
coincides with the right radical
$$
\operatorname{r.rad}=\{\F \in K_0(D^b(E)) | \langle -, \F \rangle =0 \},
$$
and we call it the radical of the Euler form and denote it $\rad$\,.

Let us define the charge of $\E \in \Coh(E)$ as
$$
Z(\E)=\left(
       \begin{matrix}
         \rk(\E) \\
         \deg(\E)
       \end{matrix}
     \right)\,.
$$
The next proposition is crucial for the description of the automorphism group of the derived category $D^b(E)$ of an elliptic curve.
\begin{Prop}
The charge Z induces an isomorphism of abelian groups
$$
Z : K_0(E)/\rad \to \Z^2\,.
$$
Moreover, the charges of the structure sheaf $\O$ and of the residue field $k(x)$ at a point $x \in E$ form the standard basis for the lattice $\Z^2$
\begin{align*}
Z(\O)=\left(
       \begin{matrix}
         1 \\
         0
       \end{matrix}
     \right), & \qquad    Z(k(x))=\left(
       \begin{matrix}
         0 \\
         1
       \end{matrix}
     \right)\,.
\end{align*}
\end{Prop}
When we need a matrix of an endomorphism of $K_0(E)/\rad$ we compute it in the standard basis for the lattice $K_0(E)/\rad \cong \Z^2$\,.

Any automorphism of the derived category $D^b(E)$ induces an automorphism of the Grothendieck group and thus an automorphism of $K_0(E)/\rad$\,. Such automorphisms preserve the Euler form, therefore, we get a group homomorphism
$$
\pi : Aut(D^b(E)) \to \operatorname{SP}_2(\Z) \cong \operatorname{SL}_2(\Z)\,.
$$

Moreover, we observe that on an elliptic curve $E$ the objects $\O$ and $k(x)$ are spherical, thus we can define two endofunctors:
\begin{align*}
\mathbb{A}=T_{\O}, & \qquad \mathbb{B}=T_{k(x)}\,.
\end{align*}
Another description for the functor $\mathbb{B}$ was given in \cite{ST01}:
\begin{Lem}
There is an isomorphism of functors $\mathbb{B} \cong \O(x) \otimes -$\,.
\end{Lem}
The images of these automorphisms in $\operatorname{SL}_2(\Z)$ can be easily computed:
\begin{align*}
A=\pi(\mathbb{A})=\left(\begin{array}{cc}
                         1 & -1 \\
                         0 & 1
                         \end{array}
                  \right), & \qquad B=\pi(\mathbb{B})=\left(\begin{array}{cc}
                         1 & 0 \\
                         1 & 1
                         \end{array}
                  \right)\,.
\end{align*}

The category $\MCM_{\gr}(R)$ has a natural autoequivalence -- the shift of internal degree $(1)$\,. By Orlov's equivalence it induces some autoequivalence $\sigma$ of $D^b(E)$
$$
\sigma = \Phi^{-1} \circ (1) \circ \Phi: D^b(E) \to D^b(E)\,.
$$

In \cite{KMVdB11} (lemma $4.2.1$) the functor $\sigma$ was completely described. Here we use that description to express $\sigma$ in terms of the spherical twists $\A$ and $\B$\,. The result depends on the degree $n$ of the elliptic normal curve
\begin{Lem}
There is an isomorphism of functors $\sigma \cong \B^n \circ \A\,.$
\end{Lem}

Generally speaking, Orlov's equivalence depends on choosing a cutoff $i \in \Z$, thus $\sigma$ also depends on the cutoff parameter $i$\,. We choose the cutoff parameter $i=0$ in this paper. Note that in \cite{KMVdB11} a formula for $\sigma$ is given for any $i \in \Z$\,.

\section{The Main Formula for Graded Betti numbers}

In this section $R$ is a commutative connected Gorenstein algebra $R$ over a field $k$\,. We assume that the Gorenstein parameter $a$ of $R$ is is zero, $a=0$\,. Orlov's equivalence
$$
\Phi : D^b(\Proj(R)) \to \MCM_{\gr}(R)
$$
of triangulated categories can be used to answer questions about MCM modules in terms of the geometry of $\Proj(R)$\,. In particular, we use it to derive a general formula for the graded Betti numbers of MCM modules.

Let $M$ be a finitely generated graded $R$-module, and let $P^\bullet \to M$ be a minimal free resolution of $M$ over $R$
$$
\dots \rightarrow P^1 \rightarrow P^0 \rightarrow M \rightarrow 0,
$$
where each term $P^i$ of the resolution is a direct sum of free modules with generators in various degrees
$$
P^i \cong \bigoplus_j R(-j)^{\beta_{i,j}}\,.
$$
The exponents $\beta_{i,j}$ are positive integers, that are called the (graded) Betti numbers. The Betti numbers contain information about the shape of a minimal free resolution.

We start with the simple observation that the graded Betti numbers are given by the formula
$$
\beta_{i,j}(M)=\dim \Ext^i(M, k(-j))\,.
$$
The extension groups $\Ext^i(M, k(-j))$ are isomorphic to the cohomology groups of the complex $\Hom(P^\bullet, k(-j))$, but the resolution $P^\bullet$ is minimal, therefore, differentials in the complex $\Hom(P^\bullet, k(-j))$ vanish.

In Orlov's equivalence we deal with a stable category of MCM modules, so our next step is to express the Betti numbers in terms of dimensions of stable extension groups.
\begin{Lem}
Let $M$ be a finitely generated MCM module. Then the graded Betti numbers are equal to the dimensions of the following stable extension groups
$$
\beta_{i,j}(M)=\dim\underline{\Ext}^i(M, k^{st}(-j)),
$$
for $i \geq 0$\,.
\end{Lem}
\begin{proof}
By Lemma \ref{CR}
$$
\Ext^i(M, k(-j)) \cong \underline{\Ext}^i(M, k(-j)),
$$
for $i \geq 0$ and any MCM module $M$ without free summands\,. Next by Lemma \ref{FGM} we have
$$
\underline{\Ext}^i(M, k(-j))=\underline{\Ext}^i(M, k^{st}(-j))\,.
$$
Combining these two lemmas we get the result.
\end{proof}

\begin{Rem}
The formula above also can be used to compute graded Betti numbers of an MCM module $M$ for $i < 0$\,. This means that we replace the projective resolution of $M$ by a complete minimal resolution
$$
\CR(M)^\bullet = \ldots \leftarrow P_{-2} \leftarrow P_{-1} \leftarrow P_0 \leftarrow P_1 \leftarrow P_2 \leftarrow \ldots,
$$
and we use such a complete resolution to extend the definition of $\Ext^i(M, k(-j))$ for $i<0$\,.
\end{Rem}

This is a useful reformulation, because stable extension groups can be computed on the geometric side of Orlov's equivalence. We have the following result:
\begin{Thm}\label{main}
The graded Betti numbers of an MCM module $M$ are given by
$$
\beta_{i,j}(M)=\dim \Hom_{D^b(\Proj(R))}(\Phi^{-1}(M), \sigma^{-j}(\Phi^{-1}(k^{st}))[i])\,.
$$
\end{Thm}
\begin{proof}
This formula follows immediately from the previous lemma and the definition of the functor $\sigma: D^b(\Proj(R)) \to D^b(\Proj(R))$\,.
\end{proof}

\section{The Cone over a plane cubic}\label{pc}

\subsection{Betti numbers of MCM modules}

Let us fix a smooth elliptic curve $E$ over an algebraically closed field $k$ with distinguished point $x \in E$, and choose the invertible sheaf $\L=\O(1)=\mathcal{O}(3x)$\,. The complete linear system of $\L$ provides an embedding of $E$ into the projective plane $\mathbb{P}(W^\vee)$:
$$
E \subset \mathbb{P}(W^\vee),
$$
where $W=H^0(E, \L)$\,. The homogeneous coordinate ring of this embedding
$$
R_E=\bigoplus_{i \geq 0} H^0(E, \L^{\otimes i})
$$
is a hypersurface ring given be a cubic polynomial $f$, namely $R_E=k[x_0, x_1, x_2]/(f)$\,.

Recall that the endofunctor $\sigma$ on the geometric side represents the endofunctor of internal degree shift $(1)$ on the stable category of graded MCM modules
$$
\sigma = \Phi^{-1} \circ (1) \circ \Phi\,.
$$
In the case of a plane cubic the functor $\sigma$ is the following composition of spherical twist functors:
$$
\sigma=\B^3 \circ \A,
$$
where $\B=T_{k(x)} \cong \mathcal{O}_E(x) \otimes -$, $\A \cong T_{\mathcal{O}_E}$\,.

We start with the following elementary, but useful lemma.
\begin{Lem}
There is a natural isomorphism of functors
$$
\sigma^3 \cong [2]
$$
on the derived category $D^b(E)$ of coherent sheaves on the elliptic curve $E$\,.
\end{Lem}
\begin{proof}
It follows immediately from the corresponding natural isomorphism
$$
(3) \cong [2]
$$
of functors on the stable category $\MCM_{\gr}(R_E)$ of MCM modules over $R_E$\,.
\end{proof}

The isomorphism $(3) \cong \cosyz^2$ is a graded version of the statement that an MCM module over a hypersurface ring has a $2$-periodic resolution.

Let us denote $\F=\Phi^{-1}(M)$, then the main formula for graded Betti numbers applied to the elliptic curve $E$ takes the following form.

\begin{Thm}
There is an isomorphism
$$
\Phi^{-1}(k^{st}) \cong \O[1]
$$
in the category $D^b(E)$\,. The graded Betti numbers $\beta_{i,j}$ of the MCM module $M=\Phi(\F)$ are given by the formula
$$
\beta_{i,j}(M)=\dim \Hom_{D^b(E)}(\F, \sigma^{-j}(\O[1])[i])\,.
$$
Moreover, the graded Betti numbers have the following periodicity property $\beta_{i+2,j}=\beta_{i,j+3}$\,.
\end{Thm}
\begin{proof}
We need to show that
$$
\Phi(\O) = \st(\bigoplus_{i \geq 0} \RHom(\O, \O(i))) \cong k^{st}[-1]\,.
$$
Set $C =\bigoplus_{i \geq 0} \RHom(\O, \O(i))$\,. By the Grothendieck vanishing theorem the complex $C$ has cohomology only in degrees $0$ and $1$:
$$
H^0(C) \cong \bigoplus_{i \geq 0} \Hom(\O, \O(i)) \cong \bigoplus_{i \geq 0} H^0(E, \O(i)) \cong R_E,
$$
and
$$
H^1(C) \cong \bigoplus_{i \geq 0} \Ext^1(\O, \O(i)) \cong \bigoplus_{i \geq 0} H^1(E, \O(i)) \cong k,
$$
where we use $H^1(E, \O(i) \cong 0$ if $i>0$ and $H^1(E,\O) \cong k$\,.

We have an embedding of the zeroth cohomology
$$
0 \to H^0(C) \to C^0 \to C^1 \to \ldots
$$
We treat this as a morphism of complexes
$$
H^0(C) \to C \to C',
$$
where $C'$ is defined as a cokernel of the map $H^0(C) \to C$ above. The complex $C'$ has only non-trivial cohomology in degree $1$, namely $H^*(C') \cong k[-1]$, therefore, $C' \cong k[-1]$ in the category $D^b(gr R_{i \geq 0})$\,. Thus we get a distinguished triangle in $D^b(gr R_{i \geq 0})$:
$$
R_E \to C \to k[-1] \to R_E[1] \to \ldots
$$
Next we apply the stabilization functor $\st : D^b(gr R_{\geq i}) \to D^{\gr}_{Sg}(R)$ to this distinguished triangle:
$$
\st(R_E) \to \st(C) \to \st(k[-1]) \to \st(R_E)[1] \to \ldots
$$
Noting that $\st(R_E) \cong 0$ we get
$$
\Phi(\O) = \st(C) \cong k^{\st}[-1]\,.
$$
Therefore, the object $\Phi^{-1}(k^{st})$, that we have in the main formula can be computed explicitly for an elliptic curve:
$$
\Phi^{-1}(k^{st}) \cong \O[1]\,.
$$
The periodicity of Betti numbers $\beta_{i+2,j}=\beta_{i,j-3}$ follows from the isomorphism of functors $(3) \cong \cosyz^2$ on the stable category of graded MCM modules $\MCM_{\gr}(R_E)$\,. \end{proof}

The property of periodicity in the theorem above is well known and can be proved by more elementary methods. It implies that for computing Betti tables it is enough to compute the Betti numbers $\beta_{0,*}$ and $\beta_{1,*}$ and then to apply periodicity. Our next step is to show that not only the homological index $i$ can be restricted to $i=0,1$, but the window of non-zero values for the inner index $j$ is also very small.

For this purpose we compute the iterations of the functor $\sigma$ on $D^b(E)$\,. Set
$$
V_j=\sigma^{j}(\mathcal{O}_E[1]),
$$
for $j \in \mathbb{Z}$\,.

We start with the simple observation that $\A(\O)=\O$, which implies $V_1=\O(1)[1]$\,. In the following lemma we explicitly compute $V_2$\,.
\begin{Lem}
For the object $V_2$ the following is true:
$$
V_2=\sigma(V_1) \cong K'[2],
$$
where $K' \cong \Omega_{\mathbb{P}^2}^1(2) \Bigl|_E$, in particular $K'$ is a locally free sheaf and its charge is
$$
Z(K')=\left(
       \begin{matrix}
         2 \\
         3
       \end{matrix}
     \right)\,.
$$
\end{Lem}
\begin{proof}
For $\A(V_1)$ we have an exact triangle
$$
\RHom(\O, V_1) \otimes \O \xrightarrow{\ev} V_1 \to \A(V_1) \to \ldots
$$
Note that $\RHom(\O, V_1) \otimes \O \cong W \otimes \O [1]$, the evaluation map $\ev$ is surjective, and
we have a short exact sequence of locally free sheaves on the elliptic curve:
$$
0 \to K \to W \otimes \O \to \O(3x) \to 0\,.
$$
We see that
$$
\A(V_1)=K[2],
$$
where $K$ is a locally free sheaf, and it is easy to see that its charge is
$$
Z(K)=\left(
       \begin{matrix}
         2 \\
         -3
       \end{matrix}
     \right)\,.
$$
Moreover, comparing short exact sequence above with the Euler sequence on $\mathbb{P}^2$
$$
0 \to \Omega_{\mathbb{P}^2}^1(1) \to W \otimes \mathcal O_{\mathbb{P}^2} \to \mathcal O(1) \to 0,
$$
we can conclude that $K \cong \Omega_{\mathbb{P}^2}^1(1) \Bigl|_E$\,. Therefore, for $V_2$ we have
$$
V_2=\B^3(\A(V_1)) \cong K \otimes \O(1) [2]\,.
$$
If we define $K'$ as
$$
K' = K \otimes \O(1) \cong \Omega_{\mathbb{P}^2}^1(2) \Bigl|_E,
$$
the object $V_2$ is given by
$$
V_2 = K'[2],
$$
as desired.
\end{proof}

For the rest of this section we use shorthand notation $K'$ for $\Omega_{\mathbb{P}^2}^1(2) \Bigl|_E$ .

Knowing $V_0$, $V_1$ and $V_2$ and using $\sigma^3 \cong [2]$, we can compute all other objects $\sigma^{-j}(\O[1])=V_j$\,. For instance,
\begin{equation*}
\begin{split}
V_{-1}=\sigma^{-1}(\O[1])&=\sigma^{-1} \circ \sigma^3 (\O[-1])=\\
&=\sigma^2 (\O[-1])=\sigma^2(\O[1])[-2]=V_2[-2]=K',
\end{split}
\end{equation*}
and
$$
V_{-2}=\sigma^{-2}(\O[1])=\sigma^{-2} \circ \sigma^3 (\O[-1])=\sigma (\O)[-1]=\O(3x)[-1]\,.
$$
We summarize our results in the following proposition.

\begin{Prop}
The objects $V_j=\sigma^j(\O[1])$ are completely determined by
$$
V_0=\O[1], \quad V_1=\O(3x)[1], \quad V_2=K'[2],
$$
and
$$
V_{3i+j}=V_j[2i],
$$
for $i \in \Z$ and $j \in \{0,1,2\}$\,.
\end{Prop}

Explicit formulae for $V_j$ can be used to prove the following vanishing property of Betti numbers.
\begin{Prop}
If $\mathcal{F}$ is an indecomposable object of $D^b(E)$ concentrated in cohomological degree $l$, then the only potentially non-vanishing Betti numbers $\beta_{0,*}$  are
$\beta_{0, l-1}$, $\beta_{0, l}$, $\beta_{0, l+1}$, and the only potentially non-vanishing Betti numbers $\beta_{1,*}$ are $\beta_{1, l+1}$, $\beta_{1, l+2}$, $\beta_{1, l+3}$\,.
\end{Prop}
\begin{proof}
The proposition follows from the formulae for $V_j$ and the observation that the groups $\Hom_{D^b(E)}(X,Y[i])$ for vector bundles $X$ and $Y$ on $E$ vanish unless $i=0,1$\,.
\end{proof}

Without loss of generality we can restrict our attention to the objects concentrated in cohomological degrees $0$ or $1$\,. Moreover, the Betti table of the module $M(1)$ is a shift of the Betti table of $M$\,. On the geometric side it means that it is enough to study Betti numbers up to the action of the functor $\sigma$\,. The action of $\sigma$ induces an action on $K_0(E)$ and $\Z^2 \cong K_0(E)/\rad$\., The matrix of $[\sigma]$ in the standard basis $\{Z(\O), Z(k(x))\}$ for $\Z^2 \cong K_0(E)/\rad$ is given by
$$
[\sigma]=B^3 \circ A=\left(
                       \begin{array}{cc}
                         1 & 0 \\
                         3 & 1 \\
                       \end{array}
                     \right)  \left(
                                     \begin{array}{cc}
                                       1 & -1 \\
                                       0 & 1 \\
                                     \end{array}
                                   \right)=\left(
                                             \begin{array}{cc}
                                               1 & -1 \\
                                               3 & -2 \\
                                             \end{array}
                                           \right).
$$

The action of the cyclic group of order $3$ generated by $[\sigma]$ can be extended to $\mathbb{R}^2$, that contains $\mathbb{Z}^2\cong K_0(E)/\rad$ as the standard lattice. For the real plane $\mathbb{R}^2$ we continue to use the coordinates $(r, d)$\,. A fundamental domain for this action can be chosen in the following form:
\begin{gather*}
r>0, \\
0 \leq d < 3r\,.
\end{gather*}

These conditions are illustrated in the Figure $1$ below. On it, dots represent points of the integer lattice $\mathbb{Z}^2\cong K_0(E)/\rad$ in the $(r,d)$ plane. The line $d=3r$ is dotted because it is not included in the fundamental domain. The meaning of the second line $3r=2d$ in the fundamental domain and the double points on the lines $d=0$ and $3r=2d$ will be made clear in the next proposition.

\begin{figure}[H]
\centering
\includegraphics{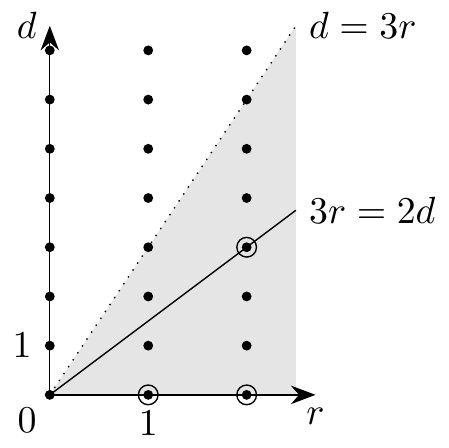}

\captionsetup{labelsep=period}
\caption{Fundamental domain in $(r,d)$-plane.}
\end{figure}

In other words, if we are interested in all possibilities for Betti tables up to shift $(1)$ it is enough to consider sheaves with charges in this fundamental domain. In particular, it is enough to consider only vector bundles on $E$\,. The graded Betti numbers of MCM modules can now be expressed as dimensions of cohomology groups of vector bundles on $E$\,. Computing dimensions of cohomology groups in the case at hand is a simple exercise.

\begin{Prop}\label{cubic}
Let $\mathcal{F}$ be an indecomposable vector bundle with charge
$
Z(\mathcal{F})=\left(
       \begin{matrix}
         r \\
         d
       \end{matrix}
     \right)
$
in the fundamental domain. For points on the two rays $d=0$ and $3r=2d$ inside the fundamental domain there are two possibilities for the Betti numbers, and for any other point in the fundamental domain the Betti numbers are completely determined by the discrete parameters $r$ and $d$\,. Moreover, Betti numbers of an indecomposable MCM module $\Phi(\mathcal{F})$ can be expressed as dimensions of cohomology groups on the elliptic curve in the following way:
\begin{align*}
&\beta_{0,-1}=0\,. \\
&\beta_{0,0}=
\begin{cases}
$1$, &\text{if $d=0$ and $\F \cong F_r$,} \\
$d$, &\text{otherwise.} \\
\end{cases} \\
&\beta_{0,1}=
\begin{cases}
$0$, &\text{if $3r<2d$,} \\
$1$, &\text{if $3r=2d$ and $\F \cong \sigma^{-1}(F_{r/2}[-1])$,} \\
$3r-2d$, &\text{otherwise.} \\
\end{cases} \\
&\beta_{1,1}=
\begin{cases}
$0$, &\text{if $3r>2d$,} \\
$1$, &\text{if $3r=2d$ and $\F \cong \sigma^{-1}(F_{r/2}[-1])$,} \\
$2d-3r$, &\text{otherwise.} \\
\end{cases}\\
&\beta_{1,2}=3r-d\,.\\
&\beta_{1,3}=
\begin{cases}
$1$, &\text{if $d=0$ and $\F \cong F_r$,} \\
$0$, &\text{otherwise.} \\
\end{cases}
\end{align*}

\end{Prop}
\begin{proof}
First we recall a well-known formula for the degree of a tensor product of two vector bundles $\F_1$ and $\F_2$ on a smooth projective curve:
$$
\deg(\F_1 \otimes \F_2) = \deg(\F_1) \rk(\F_2)+\rk(\F_1) \deg(\F_2)\,.
$$
The computation of the Betti number $\beta_{0,-1}$ is straightforward:
$$
\beta_{0,-1}=\dim \Hom_{D^b(E)}(\mathcal{F},\O(3x)[1])=\dim H^1(E, \F^\vee\otimes \O(3x))=0,
$$
because $\deg(\F^\vee\otimes \O(3x))=3r-d>0$, for a pair $(r,d)$ in the fundamental domain.

To compute $\beta_{0,0}$ we use Serre duality and get two cases:
$$
\beta_{0,0}=\dim \Hom_{D^b(E)}(\mathcal{F},\O[1])=\dim H^0(E,\F)=\\
\begin{cases}
$1$, &\text{if $d=0$ and $\F \cong F_r$,} \\
$d$, &\text{otherwise.} \\
\end{cases}
$$
By analogous arguments we get the Betti numbers $\beta_{1,2}$ and $\beta_{1,3}$:
$$
\beta_{1,2}=\dim \Hom_{D^b(E)}(\mathcal{F},\O(3x))=\dim H^0(E, \F^\vee \otimes \O(3x))=3r-d,
$$
while for $\beta_{1,3}$ we again get two cases:
$$
\beta_{1,3}=\dim \Hom_{D^b(E)}(\mathcal{F},\O)=\dim H^1(E, \F)=\\
\begin{cases}
$1$, &\text{if $d=0$ and $\F \cong F_r$,} \\
$0$, &\text{otherwise.} \\
\end{cases}
$$
In order to compute $\beta_{0,1}$ we use the fact that the left adjoint functor for the autoequivalence $\sigma^{-1}$ is $\sigma$\,. Thus,
$$
\beta_{0,1}=\dim \Hom_{D^b(E)}(\F,\sigma^{-1}(\O[1]))=\dim \Hom_{D^b(E)}(\sigma(\F),\O[1])\,.
$$
In the next step we use the derived adjunction between the functor $\Hom$ and the tensor product
$$
\Hom_{D^b(E)}(X \otimes Z, Y) \cong \Hom_{D^b(E)}(X, Z^\vee \otimes Y),
$$
where $X$, $Y$ and $Z$ are objects in $D^b(E)$ and $Z^\vee = R \mathcal{H}om (Z,\O)$\,.
Therefore,
$$
\beta_{0,1}=\dim \Hom_{D^b(E)}(\O,\sigma(\F)^\vee [1])\,.
$$
Rank $r'$ and degree $d'$ of $\sigma(\F)^\vee [1]$ can be computed as
$$
\left(
       \begin{matrix}
         r' \\
         d'
       \end{matrix}
\right)=
-\left([\sigma]  \left(
       \begin{matrix}
         r \\
         d
       \end{matrix}
\right)\right)^{[\vee]}=
\left(
       \begin{matrix}
         d-r \\
         3r-2d
       \end{matrix}
\right),
$$
where the operation $(-)^\vee$ is induced by the sheaf-dual $(-)^\vee$, it changes the sign of the second component of the charge, while it leaves invariant the first component.

We need to analyze several cases for $\sigma(\F)^\vee [1]$\,. If $d'=3r-2d < 0$, then $r'=d-r > 0$, and $\sigma(\F)^\vee [1]$ is a vector bundle of negative degree. Therefore,
$$
\beta_{0,1}=\dim \Hom_{D^b(E)}(\O,\sigma(\F)^\vee [1])=\dim H^0(E,\sigma(\F)^\vee [1])=0\,.
$$
If $d'=3r-2d > 0$, then $r'=d-r$ is not necessarily positive. Thus $\sigma(\F)^\vee [1]$ is a vector bundle or translation of a vector bundle by $[1]$ depending on the sign of $r'$\,. If $\sigma(\F)^\vee [1]$ is a vector bundle,
$$
\beta_{0,1}=\dim \Hom_{D^b(E)}(\O,\sigma(\F)^\vee [1])=\dim H^0(E,\sigma(\F)^\vee [1])=d'=3r-2d\,.
$$
If $\sigma(\F)^\vee [1]$ is a translation of a vector bundle, in other words $\sigma(\F)^\vee$ is a vector bundle with $\rk \sigma(\F)^\vee = -r'>0$, $\deg \sigma(\F)^\vee = -d'<0$\,. We obtain the same answer for $\beta_{0,1}$ as in the previous case:
$$
\beta_{0,1}=\dim \Hom_{D^b(E)}(\O,\sigma(\F)^\vee [1])=\dim H^1(E,\sigma(\F)^\vee)=d'=3r-2d\,.
$$

Finally, we need to analyze the case $d'=3r-2d = 0$\,. The solutions of the diophantine equation $3r-2d=0$ are of the form $r=2l$, $d=3l$, where $l \in \Z$, and we are only interested in positive solutions. Therefore, we consider only $l \geq 1$\,. In particular, rank $\rk \sigma(\F)^\vee [1]=r'=3l-2l=l$\,.

If $\sigma(\F)^\vee [1]$ is isomorphic to the Atiyah bundle,
$$
\sigma(\F)^\vee [1] \cong F_{l},
$$
we have
$$
\beta_{0,1}=\dim \Hom_{D^b(E)}(\O,F_l)=1\,.
$$

The condition $\sigma(\F)^\vee [1] \cong F_{l}$ determines $\F$ uniquely up to isomorphism. Indeed, this condition is equivalent to
$$
\sigma(\F) \cong (F_l[-1])^\vee \cong F_l[1],
$$
where the second isomorphism is based on the fact that Atiyah bundles are self-dual $F_l^\vee \cong F_l$\,. Therefore, for $\F$ we have
$$
\F \cong \sigma^{-1}(F_l[-1])\,.
$$

If $\sigma(\F)^\vee [1]$ is not isomorphic to an Atiyah bundle, the Betti number $\beta_{0,1}$ vanishes $\beta_{0,1}=3r-2d=0$\,.

Combining all cases we get
$$
\beta_{0,1}=\begin{cases}
0, &\text{if $3r<2d$,} \\
1, &\text{if $3r=2d$ and $\F \cong \sigma^{-1}(F_{r/2}[-1])$,} \\
3r-2d, &\text{otherwise.} \\
\end{cases}
$$

Analyzing in the same way all cases for $\beta_{1,1}$ we get the formula
$$
\beta_{1,1}=\dim \Hom_{D^b(E)}(\mathcal{F},V_{-1}[1])=
\begin{cases}
0, &\text{if $3r>2d$,} \\
1, &\text{if $3r=2d$ and $\F \cong \sigma^{-1}(F_{r/2}[-1])$,} \\
2d-3r, &\text{otherwise.} \\
\end{cases}
$$
\end{proof}

\begin{Rem}
\begin{enumerate}
  \item The line $3r=2d$ in Figure $1$ above represents a line where the special cases $\F \cong \sigma^{-1}(F_{r/2}[-1])$ can occur.
  \item Encircled points in Figure $1$ represent charges in the fundamental domain for which the Betti tables are not completely determined by the discrete parameters $r$ and $d$\,.
  \item We denote $S_l=\sigma^{-1}(F_{l}[-1])$ the vector bundles corresponding to the special cases $3r=2d$\,. Note that $\rk S_l = 2l$, $\deg S_l = 3l$\,.
  \item The condition $3r=2d$ listed for $\beta_{0,1}$ and $\beta_{1,1}$ in the proposition is superfluous, as $\F \cong \sigma^{-1}(F_{r/2}[-1])$ implies it. But we keep it for clarity of exposition.
\end{enumerate}
\end{Rem}

The results of Proposition \ref{cubic} can be expressed in a more readable form as Betti tables.
\begin{table}[H]
\begin{center}
\begin{tabular}{|r|r|r|}
\hline
$j$ & $i=0$ & $i=1$ \\
\hline
$0$ & $\beta_{0,0}$ & $\beta_{1,1}$ \\
$1$ & $\beta_{0,1}$ & $\beta_{1,2}$ \\
$2$ & $0$ & $\beta_{1,3}$ \\
\hline
\end{tabular}
\caption{}
\end{center}
\end{table}
We keep only the part of the Betti table corresponding to the cohomological indices $i=0,1$, as for other values of $i$ the table can be continued by $2$-periodicity.

In the corollary below we continue to assume that the point $(r,d)$ is in the fundamental domain.
\begin{Cor}
There are four cases of Betti tables of indecomposable MCM modules over the ring $R_E$\,. We summarize these cases in the following table
\begin{table}[H]
\begin{center}
\begin{tabular}{|r|r|r||r|r||r|r||r|r|}
\hline
& \multicolumn{2}{|c||}{$F_r$} & \multicolumn{2}{|c||}{$S_l$} & \multicolumn{2}{|c||}{$3r-2d > 0$} & \multicolumn{2}{|c|}{$3r-2d \leq 0$}\\
\hline
$j$ & $i=0$ & $i=1$ & $i=0$ & $i=1$ & $i=0$ & $i=1$ & $i=0$ & $i=1$ \\
\hline
$0$ & $1$ & $0$ & $3l$ & $1$ & $d$ & $0$ & $d$ & $2d-3r$ \\
$1$ & $3r$ & $3r$ & $1$ & $3l$ & $3r-2d$ & $3r-d$ & $0$ & $3r -d$ \\
$2$ & $0$ & $1$ & $0$ & $0$ & $0$ & $0$ & $0$ & $0$ \\
\hline
\end{tabular}
\caption{}
\end{center}
\end{table}
In particular, we see that the discrete parameters of the vector bundle $\F$ can be read off the Betti table of the MCM module $M=\Phi(\F)$ in all four cases.
\end{Cor}
\begin{proof}
This follows immediately from the proposition \ref{cubic}.
\end{proof}

The most convenient way to present the results is in the form of complete resolutions of indecomposable MCM modules. Numbers in the vertices of the diagrams below represent ranks of free modules occurring in the complete resolutions.

\begin{Cor}\label{CoRef}
The Betti diagram of a complete resolution of an indecomposable MCM module over the homogeneous coordinate ring $R_E$ has one of the following four forms, that occur in two discrete and two continuous families. On the diagrams below $[0]$ indicates cohomological degree $0$ of a complex.
\begin{enumerate}
\item The first discrete family $F_r$, where $r$ is a positive integer.
$$
\xymatrix{
\ldots& 3r \ar[l]& 3r \ar[l]& 1\ar[l] \ar@{--}[d]\\
& & 1 \ar[ul]& 3r \ar[l] \ar[ul] \ar@{--}[dd]& 3r\ar[l] \ar[ul]& 1 \ar[l]\\
& & & & 1 \ar[ul] \ar[uul]& 3r \ar[l] \ar[ul]& 3r \ar[l] \ar[ul]& \ldots \ar[l]\\
& & & [0]
}
$$
\item The second discrete family $S_l$, where $l$ is a positive integer.
$$
\xymatrix{
\ldots& 3l \ar[l]& 1 \ar[l]& \\
& 1 \ar[ul]& 3l \ar[ul] \ar[l]& 3l \ar[l] \ar[ul] \ar@{--}[d] \ar@{--}[u]& 1 \ar[l]\\
& & & 1 \ar[ul] \ar[uul] \ar@{--}[d]& 3l \ar[l] \ar[ul]& 3l \ar[l] \ar[ul]& 1 \ar[l]& \ldots \ar[l]\\
& & & [0]
}
$$
\item The first continuous family $G_\lambda(r,d)$\,. Elements in the family $G_\lambda(r,d)$ are parameterized by a pair of integers $(r,d)$, satisfying conditions $r>0$,$d \geq 0$, $3r-2d > 0$, and a point $\lambda \in E$\,.
$$
\xymatrix{
\ldots& d \ar[l]& \\
& 3r-2d \ar[ul]& 3r-d \ar[ul] \ar[l] \ar@{--}[dd] \ar@{--}[u]& d \ar[l]\\
& & & 3r-2d \ar[ul]& 3r-d \ar[l] \ar[ul]& \ldots \ar[l]\\
& & [0]
}
$$
\item The second continuous family $H_\lambda(r,d)$\,. Elements in the family $H_\lambda(r,d)$ are parameterized by a pair of integers $(r,d)$, satisfying conditions $r>0$, $3r-d > 0$, $3r-2d \leq 0$, and a point $\lambda \in E$\,.
$$
\xymatrix{
\ldots& d \ar[l]& 2d-3r \ar[l] \ar@{--}[d]\\
& & 3r-d \ar[ul] \ar@{--}[dd]& d \ar[ul] \ar[l]& 2d-3r \ar[l]\\
& & & & 3r-d \ar[ul]& \ldots \ar[l] \ar[ul]\\
& & [0]
}
$$
\end{enumerate}
\end{Cor}
\begin{proof}
This is just a $2$-periodic extension of the Betti tables we obtained above.
\end{proof}

\begin{Cor}
The two discrete families and the two continuous families are syzygies of each other, respectively. More precisely,
$$
\syz \Phi(F_r) \cong \Phi(S_{r})(1),
$$
for the discrete families and
$$
\Phi(H_\lambda(r,d))(1) \cong \syz \Phi(G_\lambda(d-r,2d-3r))\,.
$$
for the continuous families.
\end{Cor}
\begin{proof}
For the discrete families the given formula is the same as was used to define the vector bundles $S_r=\sigma^{-1}(F_{r}[-1])$\,.
For the continuous families we have the isomorphism of vector bundles
$$
\sigma (H_\lambda(r,d)[1]) = G_\lambda(d-r,2d-3r),
$$
which is equivalent to the second formula of the corollary.
\end{proof}

\begin{Rem}
If we want to study Betti tables not only up to shifts in internal degree but also up to taking syzygies we can restrict to a smaller fundamental domain. More precisely, we need to change the fundamental domain of the cyclic group of order $3$ generated by $[\sigma]$ to the fundamental domain of the cyclic group of order $6$ generated by $-[\sigma]$\,. The latter fundamental domain can be chosen in the form
\begin{gather*}
r>0, \\
3r > 2d \geq 0\,.
\end{gather*}
\end{Rem}

\subsection{The Hilbert Series and numerical invariants of MCM modules}

Our next goal is to compute the Hilbert series $H_M(t)=\sum_i \dim (M_i)t^i$ in all four cases. Let
$$
0 \leftarrow M \leftarrow F \leftarrow G \leftarrow F(-3) \leftarrow \ldots
$$
be a minimal $2-$periodic free resolution of the MCM module $M$ over the ring $R_E$\,. Then the Hilbert series of $M$ can be computed as
$$
H_M(t)=\frac{1}{1-t^3}\left(H_F(t)-H_G(t)\right)\,.
$$
We use the isomorphisms $F \cong \oplus_j R(-j)^{\beta_{0,j}}$ and $G \cong \oplus_j R(-j)^{\beta_{1,j}}$ to conclude that
$$
H_F(t)=\left(\sum_j \beta_{0, j} t^j \right) H_{R_E}(t),
$$
and
$$
H_G(t)=\left(\sum_j \beta_{1, j} t^j \right) H_{R_E}(t)\,.
$$
It is easy to see that
$$
H_{R_E}(t)=\frac{1+t+t^2}{(1-t)^2}\,.
$$
Then the Hilbert series of $M$ can be written as
$$
H_M(t)=\frac{H_{R_E}(t)}{1-t^3} \left( \sum_j \beta_{0, j} t^j - \sum_j \beta_{1, j} t^j \right)= \frac{1}{(1-t)^2} \left(\frac{B(t)}{1-t}\right)=\frac{1}{(1-t)^2} P(t),
$$
where
$$
B(t)=\sum_j \beta_{0, j} t^j - \sum_j \beta_{1, j} t^j\,.
$$
Note that $B(1)=0$ and, therefore, $P(t)=\frac{B(t)}{1-t}$ is a polynomial for a module $M$ generated in positive degrees ($P(t)$ is a Laurent polynomial in general). Moreover, the multiplicity of an MCM module $M$ is given by
$$
e(M)=P(1)\,.
$$
Consequently, we can read off the Betti tables the Hilbert series and, in particular, the multiplicity in all cases. The multiplicity of an MCM module $M$ is related to the multiplicity of the ring $e(M)=\rk(M)e(R)$\,. Note that in this formula $\rk$ means rank of the module $M$ (not to be confused with the rank of the corresponding vector bundle).  In the case of the homogeneous coordinate ring of a smooth cubic $e(R)=3$, thus, ranks of MCM modules also can be extracted from the Hilbert series. The number of generators $\mu(M)$ of a module $M$ can be easily computed as a sum of Betti numbers
$$
\mu(M)=\sum_j \beta_{i,j},
$$
for any $i \in \Z$\,. In particular, we can choose $i=0$\,.

Therefore, the numerical invariants of an MCM module $M=\Phi(\F)$ can be expressed in terms of the discrete invariants $r=\rk(\F)$ and $d=\deg(\F)$ of the vector bundle $\F$\,. We present the result in the form of a table:
\begin{table}[H]
\begin{center}
\begin{tabular}{|r|r|r|r|r|}
\hline
case & $P(t)$ & $e(M)$ & $\mu(M)$ & $\rk(M)$ \\
\hline
$F_r$ & $1+t+3rt+t^2$ & $3r+3$ & $3r+1$ & $r+1$ \\
$S_l$ & $3l+3lt$ & $6l$ & $3l+1$ & $2l$ \\
$3r-2d \geq 0$ & $d+(3r-d)t$ & $3r$ & $3r-d$ & $r$ \\
$3r-2d < 0$ & $d+(3r-d)t$ & $3r$ & $d$ & $r$ \\
\hline
\end{tabular}
\caption{\label{NumCubic}}
\end{center}
\end{table}

We finish this section with a discussion of maximally generated MCM modules or Ulrich modules over the ring $R_E$\,. Recall that a module over a Cohen-Macaulay ring is Cohen-Macaulay if and only if $e(M)=\operatorname{length}(M/\underline{x}M)$, for a maximal regular sequence $\underline{x}=(x_1,x_2, \ldots, x_n)$\,. Then we have
$$
\mu(M)=\operatorname{length}(M/m M)\leq \operatorname{length}(M/\underline{x}M)=e(M),
$$
where $m$ is the maximal ideal for the local case or the irrelevant ideal for the graded case. For details see \cite{Sally76} and \cite{Ulrich84}.

MCM modules for which the upper bound is achieved are called maximally generated or Ulrich MCM modules. Existence of such modules over a ring $R$ is usually a hard problem. Here we want to show that Orlov's equivalence allows us to prove the existence almost immediately for the ring $R_E$\,.

From the table above we see that if the charge of $\F$ is in the fundamental domain, then $M=\Phi(\F)$ corresponding to the third case and $d=0$ is maximally generated, unless we get a module corresponding to the Atiyah bundle. But it is convenient to work with modules generated in degree $0$ that projective cover of $M$ looks like
$$
\O^m \to M \to 0.
$$

We can archive for this Ulrich modules if we consider modules on the other boundary of the fundamental domain $d = 3r$ (boundary that we exclude from the domain).

\begin{Prop}
Let $\F$ be an indecomposable vector bundle on an elliptic curve $E$\,. If $\deg(\F)=3\rk(\F)$ and $\F$ is not an isomorphic to $\sigma(F_r)$, then the MCM module $M=\Phi(\F)$ is Ulrich and generated in degree $0$.

Conversely, every indecomposable Ulrich module over $R_E$ generated in degree $0$ is isomorphic to $\Phi(\F)$ for some indecomposable  $\F$ with $\deg(\F)=3\rk(\F)$.
\end{Prop}

\begin{proof}
The computation can be done in the same spirit as in the proof of the proposition \ref{cubic}.
\end{proof}

Note that for an Ulrich MCM module the matrix factorization is given by a matrix with linear entries and by a matrix with quadratic entries, this observation also follows from a more general result of A. Beauville for spectral curves \cite{Beau90} in the case of line bundles.

\section{Examples of Matrix factorizations of a smooth cubic}

A complete description of matrix factorizations of rank one MCM modules over the Fermat cubic can be found in \cite{LPP02}, and that for rank one MCM modules over a general elliptic curve in Weierstrass form in \cite{Galinat14}.

In this section we discuss several examples in order to illustrate our general formulae for graded Betti numbers. So far in the presentation of the ring $R_E=k[x_0, x_1, x_2]/(f)$ we did not specify the cubic polynomial $f$\,. In this section we choose $f$ to be in Hesse form:
$$
f(x_0,x_1,x_2)=x_0^3+x_1^3+x_2^3-3\psi x_0x_1x_2\,.
$$
The following theorem is classical.

\begin{Thm}
Relative to suitable homogeneous coordinates, each non-singular cubic has an equation of the form
$$
f(x_0,x_1,x_2)=x_0^3+x_1^3+x_2^3-3\psi x_0x_1x_2,
$$
and the parameter $\psi$ satisfies $\psi^3 \neq 1$\,.
\end{Thm}
\begin{proof}
See, for example, p.293 in \cite{BK86}.
\end{proof}

\subsection{A matrix factorization of \texorpdfstring{$\cosyz(k^{\st})$}{syzygies of stabilization of the residue field}}

In this section we show that a matrix factorization for $\cosyz (k^{\st})$ can be obtained as a Koszul matrix factorization. Let us recall that a Koszul matrix factorization $\{A,B\}$ of a polynomial $f$ of the form
$$
f=\sum_{i=1}^l a_i b_i
$$
is given as a tensor product of matrix factorizations \cite{Yoshino98}:
$$
\{A,B\}=\{a_1, b_1\} \otimes \{a_2,b_2\} \otimes \ldots \otimes \{a_l,b_l\}\,.
$$
For a Hesse cubic, choosing the simplest possible presentation as a sum of the products of linear and quadratic terms,
$$
x_0^3+x_1^3+x_2^3-3\psi x_0x_1x_2=x_0(x_0^2)+x_1(x_0^2)+(x_2^2-3\psi x_0 x_1)x_2,
$$
we get the following matrix factorization:
$$
\{A,B\}=\{x_0,x_0^2\} \otimes\{x_1^2,x_1\} \otimes \{x_2^2-3 \psi x_0 x_1,x_2\},
$$
where the matrices $A$ and $B$ are given by the formulae
$$
A=\begin{pmatrix}
x_0 & x_1^2 & -3 \psi x_0x_1 + x_2^2 & 0 \\
-x_1 & x_0^2 & 0 & -3 \psi x_0x_1 + x_2^2\\
-x_2 & 0 & x_0^2 & -x_1^2 \\
0 & -x_2 & x_1 & x_0
\end{pmatrix},
$$
$$
B=\begin{pmatrix}
x_0^2 & -x_1^2 & 3 \psi x_0x_1 - x_2^2 & 0 \\
x_1 & x_0 & 0 & 3 \psi x_0x_1 - x_2^2 \\
x_2 & 0 & x_0 & x_1^2 \\
0 & x_2 & -x_1 & x_0^2
\end{pmatrix}\,.
$$
We see that this matrix factorization corresponds to the first case $r=1$ in the discrete series $F_r$\,. In other words, it is a matrix factorization of $\Phi(\O)=\cosyz(k^{\st})$\,.

\subsection{Matrix factorizations of skyscraper sheaves of degree one}

Let $\lambda \in E$ be a closed point. By a skyscraper sheaf of degree one we mean a skyscraper sheaf of the form $k(\lambda)=\mathcal{O}_{E,\lambda}/m(\lambda)$\,. We assume that the point $\lambda$ can be obtained as intersection of $E$ with two lines, given as zeroes of linear forms $l_1$ and $l_2$\,. In other words, we have an inclusion of ideals:
$$
(f) \subset (l_1,l_2)\,.
$$
This implies that there are quadratic forms $f_1$ and $f_2$ such that
$$
f=l_1 f_1+l_2 f_2\,.
$$
Let
$$
A=\begin{pmatrix}
l_2 & f_1 \\
-l_1 & f_2
\end{pmatrix}\,.
$$
It is easy to see that
$$
\coker A \cong (k(\lambda)^{\st})\,.
$$

The second matrix $B$ is determined by the matrix $A$ if $\det(A)=f$ because in such case $B=fA^{-1}=\bigwedge^2(A)$\,. Therefore, the matrix $B$ is given as
$$
B=\begin{pmatrix}
f_2 & -f_1 \\
l_1 & l_2
\end{pmatrix}\,.
$$
Such matrix factorizations $\{A,B\}$ are parameterized by the $G_\lambda(1,1)$ family.

\subsection{Moore matrices as Matrix Factorizations}

We fix a line bundle $\L$ on $E$ of degree $\deg \L=3$\,. The canonical evaluation morphism $\Hom(\O,\L) \otimes \O \to \L$ is epimorphic, so we get a short exact sequence of vector bundles on $E$:
$$
0 \to \K \to \Hom(\O, \L) \otimes \O \to \L \to 0,
$$
where $\K$ is defined as the kernel of the evaluation map. Tensoring this short exact sequence with $\O(1)$, we get
$$
0 \to \K(1) \to \Hom(\O(1), \L(1)) \otimes \O(1) \to \L(1) \to 0\,.
$$
It is easy to see that the charge of the vector bundle $K(1)$ is
$$
Z(K(1))=\left(
       \begin{matrix}
         2 \\
         3
       \end{matrix}
     \right),
$$
thus $H^1(E,\K(1)) \cong 0$, and we get a short exact sequence of cohomology groups
$$
0 \to H^0(E,\K(1)) \to H^0(E,\L) \otimes H^0(E,\O(1))  \to H^0(E,\L(1)) \to 0\,.
$$

In the previous section we introduced the notation $W=H^0(E,\O(1))$\,. Let us also denote $U=H^0(E,\L)$ and $V=H^0(E,\K(1))$\,. The vector spaces $U$, $V$ and $W$ are three dimensional, as easily follows from the formula for dimensions of cohomology groups in terms of charges. The short exact sequence above thus yields the following map of vector spaces
$$
\phi: V \to U \otimes W,
$$
coming from an elliptic curve and two line bundles of degree three. Such maps were studied extensively under the name of elliptic tensors (see \cite{BP93} and references therein). We want to describe this map in more concrete terms. For this we choose a basis in each of $U$, $V$, and $W$\,. To do this we need to use the action of the finite Heisenberg group $H_3$ on these spaces.

Recall that if $\F$ is a vector bundle of degree $n>0$, then $(t_x)^*\F \cong \F$ if and only if $x \in E[n]$, where $t_x$ is a translation on $x$ in the group structure of $E$, and $E[n]$ denotes the subgroup of points of order $n$\,. The group $E[n]$ can be described explicitly as
$$
E[n] \cong \Z_n \times \Z_n\,.
$$
Another way of stating this observation is that $\F$ is an invariant vector bundle on $E$ with respect to the natural $E[n]$ action. But $\F$ is not an equivariant vector bundle. It means that only some central extension of $E[n]$ will act on the cohomology groups of $\F$\,. The universal central extension of $\Z_n \times Z_n$ by $\Z_n$ is called the Heisenberg group $H_n$\,.
$$
0 \to \Z_n \to H_n \to \Z_n \times \Z_n \to 0\,.
$$
It is a finite nonabelian group of order $n^3$\,. We denote elements mapping to generators of $\Z_n \times \Z_n$ by $\sigma$ and $\tau$, and the generating central element by $\epsilon$\,. The center of the Heisenberg group is the cyclic subgroup generated by $\epsilon$\,. The Heisenberg group $H_n$ acts on the cohomology $H^0(E, \F)$. If, for instance, $\F$ is a line bundle, then $H^0(E, \F)$ is isomorphic to a standard irreducible Schroedinger representation of $H_n$\,. For a detailed description of $H_n$, its representations and actions on the cohomology groups see \cite{Hulek86}.

The map $\phi$ constructed above is a map of $H_3$ representations. Moreover, we can think about this map as a map given by some $3 \times 3$ matrix representing a linear map from $V$ to $W$ with entries given by some linear forms in the variables $x_0$, $x_1$, $x_2$\,. Note that these coordinates (and thus some basis) for $W$ were fixed before, but we still have the freedom to choose bases for $V$ and $U$\,.

For the Schroedinger representation $V$ we choose a standard basis $\{e_i\}_{i \in \Z_3}$\,. In this basis the generators act as
\begin{gather*}
\tau e_i = \rho^i e_i,\\
\sigma e_i = e_{i-1},
\end{gather*}
where $\rho$ is a primitive third root of unity in the field $k$\,.

The images of the vectors $e_i$ are vectors $v_i$ that satisfy $\tau v_i =\rho^i v_i$ and $\sigma v_i = v_{i-1}$\,. Then it is clear that such vectors are of the form

$$
v_0=\begin{pmatrix}
         a_0x_0 \\
         a_2x_1 \\
         a_1x_2
       \end{pmatrix},
\qquad
v_1=\begin{pmatrix}
         a_2x_2 \\
         a_1x_0 \\
         a_0x_1
       \end{pmatrix},
\qquad
v_2=\begin{pmatrix}
         a_1x_1 \\
         a_0x_2 \\
         a_2x_0
       \end{pmatrix}\,.
$$

Our notation requires some explanation. The three vectors $v_0$, $v_1$ and $v_2$ represnt elements of $U \otimes W$, if we fix a basis in $U$ and interpret tensors in $U \otimes W$ as column vectors with entries from $W$. We treat the map $V \to U \otimes W$ using the same logic as a $3 \times 3$ matrix with entries from $W$. Explicitly, in the chosen bases, the map $V \to U \otimes W$  has the matrix
$$
A=\begin{pmatrix}
a_0x_0 & a_2x_2 & a_1x_1 \\
a_2x_1 & a_1x_0 & a_0x_2 \\
a_1x_2 & a_0x_1 & a_2x_0
\end{pmatrix}\,.
$$

We get thus an explicit formula for the mapping on the level of zeroth cohomology, extending it to the morphism of sheaves we get
$$
0 \to \K(1) \stackrel{A(1)}{\longrightarrow} \Hom(\O, \L) \otimes \O(1) \to \L(1) \to 0,
$$
or, equivalently,
$$
0 \to \K \stackrel{A}{\longrightarrow} \Hom(\O, \L) \otimes \O \to \L \to 0,
$$
where, of course, $A$ and $A(1)$ coincide as matrices.

There is an alternative, more geometric, interpretation of the matrix $A$. Sections of the two line bundles $\O(1)$ and $\L$ provide an embedding of the elliptic curve $E$ into the multi-projective space $\P(W) \times \P(U)$$$
\xymatrix{
& E \ar@/^-0.7pc/[ddl]_{|\L|} \ar@/^0.7pc/[ddr]^{|\O(1)|} \ar[d]^{\varphi} & \\
& \P(U) \times \P(W) \ar[dl]^{P_1} \ar[dr]_{P_2} & \\
\P(U) && \P(W)}
$$
We want to find equations of the elliptic curve $E$ in the multi-projective space $\P(U) \times \P(W)$. Equations of bi-degree $(1,1)$ span the kernel of the map
$$
\ldots \to H^0(\P(U) \times \P(W), \mathcal{O}(1,1)) \to H^0(E, \phi^* \mathcal{O}(1,1)) \to 0.
$$
Noting that $\phi^*\mathcal{O}(1,1) \cong \O(1) \otimes \L \cong \L(1)$ and $H^0(\P(U) \times \P(W), \mathcal{O}(1,1)) \cong U \otimes W$, we can conclude that we have the same short exact sequence as before
$$
0 \to V \to U \otimes W \to H^0(E, \L(1)) \to 0.
$$
In other words the same map $V \to U \otimes W$ describes $(1,1)$ equations of the elliptic curve $E$ in the multi-projective space $\P(U) \times \P(W)$, but this time we interpret this map differently: a basis vector of $V$ is mapped to a tensor in $U \otimes W$ which is one of the equations of $E$. It is easy to check that equations of bi-degree $(1,1)$ generate the bi-graded ideal of $E$. In this context the matrix $A$ appears in \cite{ATV90}, where it was shown that the parameters $a_0$, $a_1$ and $a_2$ appearing in $A$ should be interpreted as projective coordinates of a point $a$ on $E$ embedded into $\P(W)$. Moreover, it was shown that equations with parameter $a$ correspond to the situation where the embeddings $p_2 \phi$ and $p_1 \phi$ are related as follows
$$
p_1 \phi = p_2 \phi t_a.
$$

Equivalently, we can say that $\L \cong (t_a)^*\O(1)$. If we choose $x \in E$ as the origin of that abelian variety and write $\O(1) \cong \O(3x)$ then $\L \cong \O(3a)$. This is a serious disadvantage of the method: instead of parametrizing matrices $A$ by points on the elliptic curve $E$ we parametrize them by points of the degree nine isogeny:
\begin{align*}
E \to E\\
a \mapsto 3a
\end{align*}
In particular, it means that if we add any of the $9$ points of order $3$ on $E$ to $a$ we get a different matrix $A$ that corresponds to the same line bundle $\L \cong \O(3a)$.

Finally, we are going to find a complement matrix $B$ such that the pair $(A,B)$ forms a matrix factorization of $f$:
$$
AB=BA=f\operatorname{Id}.
$$
We return to the short exact sequence
$$
0 \to \K \stackrel{A}{\longrightarrow} \Hom(\O, \L) \otimes \O \to \L \to 0,
$$
and use again the evaluation map, this time for $K$. We get
$$
\ldots \to \Hom(\O, \K(1)) \otimes \O \to \K(1) \to 0.
$$
Thus $A$ is the first map in the resolution of $\L$
$$
\ldots \to \Hom(\O, K(1)) \otimes \O(-1) \stackrel{A}{\longrightarrow} \Hom(\O, \L(1)) \otimes \O \to \L \to 0.
$$
This resolution is $2$-periodic, next map is given by a matrix $B$, but to get an explicit formula for $B$ we use a different approach.

The matrix $B$ can be computed using the identity
$$
AB=\det(A)=a_0a_1a_2(x_0^3+x_1^3+x_2^3) - (a_0^3+a_1^3+a_2^3)x_0x_1x_2,
$$
if we assume, as before, that $a=[a_0:a_1:a_2]$ is a point of the elliptic curve $E \subset \P(W)$ and $a_0 a_1 a_2  \neq 0$\,. The last condition means that $a$ is not a point of order $3$ that is $a \notin E[3]$, indeed that would mean $\L \cong \O(3a) \cong \O(1)$.
Therefore, for $B=\operatorname{Adj}(A)$ we can use the following formula
$$
B=\frac{1}{a_0 a_1 a_2} \begin{pmatrix}
a_1a_2x_0^2 -a_0^2x_1x_2 & a_0a_1x_1^2-a_2^2x_0x_2 & a_0a_2x_2^2-a_1^2x_0x_1 \\
a_0a_1x_2^2-a_2^2x_0x_1 & a_0a_2x_0^2-a_1^2x_1x_2 & a_1a_2x_1^2 -a_0^2x_0x_2 \\
a_0a_2x_1^2-a_1^2x_0x_2 & a_1a_2x_2^2 -a_0^2x_0x_1 & a_0a_1x_0^2-a_2^2x_1x_2
\end{pmatrix}\,.
$$

These two matrices provide a matrix factorization of the MCM module $\Phi(\L)$\,. Note that the charge of $\L$ is not included in our fundamental domain, but the matrix factorization of a line bundle of degree $0$ will only differ by a shift in internal degree $(3)$\,. Therefore, we can consider this matrix factorization as an example that corresponds to the point $(1,0)$ in the fundamental domain. We summarize our discussion in the following
\begin{Thm} A matrix factorization of the MCM module $\Phi(\O(3a))$ is given by the matrix
$$
A=\begin{pmatrix}
a_0x_0 & a_2x_2 & a_1x_1 \\
a_2x_1 & a_1x_0 & a_0x_2 \\
a_1x_2 & a_0x_1 & a_2x_0
\end{pmatrix}\,,
$$
and the matrix
$$
B=\frac{1}{a_0 a_1 a_2} \begin{pmatrix}
a_1a_2x_0^2 -a_0^2x_1x_2 & a_0a_1x_1^2-a_2^2x_0x_2 & a_0a_2x_2^2-a_1^2x_0x_1 \\
a_0a_1x_2^2-a_2^2x_0x_1 & a_0a_2x_0^2-a_1^2x_1x_2 & a_1a_2x_1^2 -a_0^2x_0x_2 \\
a_0a_2x_1^2-a_1^2x_0x_2 & a_1a_2x_2^2 -a_0^2x_0x_1 & a_0a_1x_0^2-a_2^2x_1x_2
\end{pmatrix}\,,
$$
where $a=[a_0:a_1:a_2] \in E \subset \P(W)$, and $a$ is not a point of order $3$ on $E$ as abelian variety with $[0:-1:1]$ as origin.
\end{Thm}

We want to emphasize that we get such a nice description for these matrix factorizations when the line bundle $\L$ is described as the pullback along translation by $a$ that is $\L \cong (t_a)^*\O(1)$. On one hand, the translation parameter $a$ is determined only up to addition of a point of order $3$ and a more natural parametrization of a line bundle of degree $3$ would be $\O(2x+p)$, where $x$ is the origin and $p \in E$ is some parameter. On the other hand, this very description in terms of translation is used in the majority of the literature on abelian varieties, where in general the parameter $a$ is called a characteristic.

Though it is certainly possible to explicitly compute matrix factorizations of the Hesse cubic corresponding to the line bundle $\O(2x+p)$, using $p$ as a parameter instead of $a$, the formulas we were able to get are much more complicated. Besides, the question formulated for $\O(2x+p)$ does not look like a natural one for the Hesse canonical form.

Matrices of the same type as $A$ above are known as Moore matrices (the matrix $A$ above is an example of a $3\times 3$ Moore matrix). They have already appeared in the literature in a variety of contexts: to describe equations of projective embeddings of elliptic curves see \cite{GP98}; to give an explicit formula for the group operation on a cubic in Hesse form see \cite{Frium02} and \cite{Ranes97}; as differential in a projective resolution of the field over elliptic algebras see \cite{ATV90}. Although the matrices $A$ and $B$ above are known to be a matrix factorization of a Hesse cubic cf. \cite[example $3.6.5$]{EG10}, their relation to line bundles of degree $3$ via representations of the Heisenberg group seems never to have appeared in the literature before.

\subsection{Explicit Matrix Factorizations of skyscraper sheaves of degree one}

We want to indicate that explicit formulae of the same type as for Moore matrices can be also obtained for skyscraper sheaves of degree $1$\,.  For the next example we choose a point on the Hesse cubic $a=[a_0:a_1:a_2] \in V(f)$, so that we have an inclusion of ideals $(f) \subset (a_1x_2-a_2x_1 , a_0x_1 - a_1x_0)$\,. The first matrix in a matrix factorization $\{A,B\}$ of a skyscraper sheaf $k(a)=\mathcal{O}_{a}/m_a$ can be explicitly written as
$$
A=\begin{pmatrix}
a_1x_2-a_2x_1 & a_0x_1 - a_1x_0 \\
a_0a_2x_0^2 + a_0^2x_0x_2 - a_1^2x_1x_2 - a_2^2x_2^2 & a_2^2x_0x_2 + a_0a_2x_2^2 - a_0^2x_0^2 - a_0a_1x_1^2
\end{pmatrix},
$$
where the forms $a_1x_2-a_2x_1$ and $a_0x_1 - a_1x_0$ cut out the point $a$ on $E$\,. The matrix $B$ can be computed by the formula
$$
(a_0 a_1 a_2) AB=\det(A)=a_0a_1a_2(x_0^3+x_1^3+x_2^3) - (a_0^3+a_1^3+a_2^3)x_0x_1x_2,
$$
if we assume that $a_0 a_1 a_2  \neq 0$\,. Last condition means that $a$ is not one of the  inflection points of the Hesse pencil. The matrix $B$ is given by the following formula:
$$
B=\frac{1}{a_0 a_1 a_2}\begin{pmatrix}
a_0^2x_0^2 + a_0a_1x_1^2 - a_2^2x_0x2 - a_0a_2x_2^2 & -a_1x_0 + a_0x_1 \\
a_0a_2x_0^2 + a_0^2x_0x_2 - a_1^2x_1x_2 - a_2^2x_2^2 & a_2x_1 - a_1x_2
\end{pmatrix}\,.
$$
So, we see the same condition $a_0 a_1 a_2 \neq 0$ is necessary for explicit formulas even in the case of a skyscraper sheaf. But in comparison to the Moore's matrix in this case matrix factorizations of skyscraper sheaves of order $3$ points are still valid and could be obtained by the same procedure.

\section*{Acknowledgement}
The results of the paper are part of the results of the author's PhD thesis. I am pleased to thank my advisor Ragnar-Olaf Buchweitz, who posed the problem of computing Betti numbers in this case to me and encouraged me during my work on it. His interest in my results and numerous discussions on all stages of the project always motivated me.

\end{document}